\documentclass[12pt]{amsart}
\usepackage{amssymb, amsmath, amsthm, cancel, hyperref, color, enumerate, verbatim, mathtools}
\usepackage{multirow} 
\usepackage{stmaryrd} 
\usepackage{ytableau} 

\usepackage[margin=1in, includeheadfoot]{geometry}
\usepackage{hyperref}

\newtheorem{thm}{Theorem}
\newtheorem*{rem}{Remark}

\newtheorem{cor}[thm]{Corollary}
\newtheorem{conj}[thm]{Conjecture}
\newtheorem{prop}[thm]{Proposition}

\newcommand{\leg}[2]{\genfrac{(}{)}{}{}{#1}{#2}} 

\newcommand{\ZZ}{\mathbb{Z}}
\newcommand{\NN}{\mathbb{N}}
\newcommand{\QQ}{\mathbb{Q}}

\newcommand{\RR}{\mathbb{R}}

\usepackage{xcolor}

\title{Self-conjugate $6$-cores and quadratic forms}
\author{Michael Hanson \& Marie Jameson}
\date{September 2021}

\begin{document}

\maketitle

\begin{abstract}
In this work, we analyze the behavior of the self-conjugate 6-core partition numbers $sc_{6}(n)$ by utilizing the theory of quadratic and modular forms. In particular, we explore when $sc_{6}(n) > 0$. Positivity of $sc_{t}(n)$ has been studied in the past, with some affirmative results when $t > 7$. The case $t = 6$ was analyzed by Hanusa and Nath, who conjectured that $sc_{6}(n) > 0$ except when $n \in \{2, 12, 13, 73\}$. This inspires a theorem of Alpoge, which uses deep results from Duke and Schulze-Pillot to show that $sc_{6}(n) > 0$ for $n \gg 1$ using representation numbers of a particular ternary quadratic form $Q$.

Approximating such representation numbers involves class numbers of imaginary quadratic fields, which are directly related to values of Dirichlet $L$-functions. At present, we can only ineffectively bound these from below. This is currently the main hurdle in obtaining more explicit approximations for representation numbers of ternary quadratic forms, and in particular in showing explicit positivity results for $sc_{6}(n)$. However, by assuming the Generalized Riemann Hypothesis we are able to settle Hanusa and Nath's conjecture.
\end{abstract}

\section{Introduction \& Statement of Results}

A \emph{partition} of a nonnegative integer $n$ is a non-increasing sequence of positive integers (called parts) which sum to $n.$ Let $p(n)$ denote the number of partitions of $n$. Each partition of $n$ can be represented by a \emph{Ferrers diagram}, in which the number of cells in the $i$th row of the diagram is the $i$th part of the partition. The \emph{hook length} of a cell in the Ferrers diagram is the number of squares  below or to right of the cell (including itself). For example, (4,2,1,1) is a partition of 8 which has the following Ferrers diagram, where each cell is labeled with its hook length.
\[\ytableausetup{centertableaux}
\begin{ytableau}
7 & 4 & 2 & 1\\
4 & 1 \\
2\\
1
\end{ytableau}\]

For $t \in \NN$, a \emph{$t$-core partition} is a partition for which no hook length in the Ferrers diagram is a multiple of $t$. Denote the number of $t$-core partitions of $n$ by $c_{t}(n)$. Partition hook lengths and $t$ cores are objects of fundamental importance which appear in several areas of mathematics; for example, they have connections to the representation theory of $S_n$ and $A_n,$ congruences for $p(n),$ class numbers, and more (see, for example, \cite{GarvanKimStanton, GranvilleOno, OnoSze, Sagan}).

Here, we are interested in \emph{self-conjugate $t$-core partitions}, which are $t$-core partitions whose Ferrers diagram remains the same after switching its columns and rows. We denote the number of self-conjugate $t$-core partitions by $sc_{t}(n)$. Here also, we find deep connections between $sc_t(n)$ and other mathematical objects; for example, K.\ Ono and W.\ Raji W. Raji \cite{OnoRaji} proved that in many cases, $sc_7(n)$ is equal to a Hurwitz class number. Work of K.\ Bringmann, B.\ Kane, J.\ Males, and others also made connections between self-conjugate $t$-cores and $t$-cores, Hurwitz class numbers, and sums of squares (see, for example, \cite{Bringmann-Kane-Males, DawseySharp, MalesTripp}). Much of this work relies on connecting self-conjugate $t$-core partitions to the theory of modular forms.

One of the first questions that arise in this study is the following: when is $c_t(n)>0,$ and when is $sc_t(n)>0?$ For $c_t(n),$ this came in the form of the \emph{$t$-core positivity conjecture,} which asserts that $c_t(n)>0$ for every integer $t\geq 4.$ This was proved by A. Granville and K.\ Ono \cite{GranvilleOno}. For $sc_t(n),$ work of Baldwin et al.\ \cite{BaldwinDepwegFordKuninSze} shows that for $n\neq 2$ and $t=8$ or $t\geq 10,$ we have that $sc_t(n)>0.$ However, they note that $sc_6(n)$ is not always positive, since (for example) $sc_6(13)=0.$ After computing many values of $sc_6(n),$ Hanusa and Nath \cite{HanusaNath} made a precise conjecture regarding the positivity of $sc_6(n).$

\begin{conj}[Conjecture 3.5 of \cite{HanusaNath}] \label{conj:HanusaNath}
Let $n$ be a positive integer. Then $sc_6(n)>0$ except when $n\in \{2, 12, 13, 73\}.$
\end{conj}

A key step in this direction was made by L. Alpoge \cite{Alpoge}, who used the generating function for $sc_6(n)$
\[\sum_{n\geq 0} sc_6(n)q^n = \prod_{n\geq 1} \frac{(1-q^{2n})^2(1-q^{12n})^3}{(1-q^n)(1-q^{4n})}\]
to make the following connection between $sc_6(n)$ and representation numbers of a certain ternary quadratic form.
\begin{thm} \label{thm:alpoge}
For all $n\geq 0,$
\[sc_6(n) = \frac{1}{12} \#\{(x,y,z)\in\mathbb{Z}^3: 24n+35 = 3x^{2} + 32y^{2} + 32yz + 32z^{2}\}.\]
\end{thm}

\begin{rem}
The statement of this theorem has been adjusted to correct an error in Alpoge's calculations.
\end{rem}

This theorem is crucial because it reduces Hanusa-Nath's positivity conjecture to the question of which nonnegative integers of the form $24n+35$ are represented by the quadratic form \[Q:=3x^2+32y^2+32yz+32z^2.\]
Alpoge then applies deep results of Duke and Schulze-Pillot \cite{DSP} to this quadratic form to prove that $sc_6(n)>0$ for sufficiently large $n$, but this result is ineffective. There is at present no unconditional way to resolve Conjecture \ref{conj:HanusaNath}; counting the representations of a sufficiently large integer $n$ by $Q$ is approximated by an expression involving a class number of an imaginary quadratic field and so is intimately related to the value of a Dirichlet $L$-function, which can be ineffectively bounded from below by Siegel's theorem. 

In this work, we assume the Generalized Riemann Hypothesis (GRH) in order to prove the following statement about this quadratic form $Q.$ In order to state the theorem, we let $r_Q(n)$ denote the number of representations of $n$ by $Q,$ i.e., $r_{Q}(n) := \#\{\mathbf{x} \in \ZZ^{3} : n = Q(\mathbf{x})\}.$

\begin{thm} \label{maintechthm}
Assume the GRH for all Dirichlet $L$-functions and all modular $L$-functions and let $n$ be a positive integer. Then $r_Q(24n+35)>0$ except when $n \in \{2, 12, 13, 73\}.$
\end{thm}

This conditionally settles Hanusa-Nath's positivity conjecture.

\begin{cor} \label{thm:hanusanath}
Assume the GRH for all Dirichlet $L$-functions and all modular $L$-functions. Then Conjecture \ref{conj:HanusaNath} is true.
\end{cor}

In Section \ref{section:background} we provide a brief overview of the theory of quadratic forms and modular forms, including key results that will be used in later sections. We prove Theorem \ref{thm:alpoge} in Section \ref{section:alpoge}, and Theorem \ref{maintechthm} will be proved in Section \ref{section:hanusanath}. 


\section{Background}\label{section:background}

Here we provide a brief overview of some key concepts in the theory of quadratic forms and modular forms that we use to prove Theorem \ref{maintechthm}. See, for example, \cite{Iwaniec}.

Let $Q = Q(x_{1}, x_{2}, x_{3})$ be a positive definite integral ternary quadratic form. That is, $Q$ is a homogeneous degree-2 polynomial in three variables with coefficients in $\ZZ$ which can be expressed as
        $$Q(\mathbf{x}) = \frac{1}{2}\mathbf{x}^{t}A\mathbf{x},$$
where $A$ is a positive definite symmetric matrix with integer entries (which are even on the diagonal). We wish to understand the behavior of the function $r_{Q}(n) := \#\{\mathbf{x} \in \ZZ^{3} : n = Q(\mathbf{x})\}.$

Given $k\in \frac{1}{2}\ZZ$ and $N \in \NN$, let $M_{k}(\Gamma_{0}(N), \chi)$ and $S_{k}(\Gamma_{0}(N), \chi)$ denote respectively the spaces of modular forms and cusp forms of weight $k$, level $N$, and Nebentypus character $\chi$. When $\chi$ is trivial we drop it from notation. It is known that the \emph{theta function} associated to a ternary quadratic form $Q$,
        $$\theta_{Q}(z) := \sum_{\mathbf{x} \in \ZZ^{3}} q^{Q(\mathbf{x})} = \sum_{n\geq 0} r_{Q}(n)q^{n},$$
is a modular form of weight $3/2,$ level $2N,$ and character $\leg{2\det(A)}{\cdot},$ where $N$ is the least integer for which $NA^{-1}$ has integer entries (although $N$ may not be the minimal level).

In order to prove Theorem \ref{thm:hanusanath}, we must understand which integers are represented by $Q$, so we must understand which Fourier coefficients of $\theta_Q$ are nonzero. To do this, we follow the approach introduced by K.\ Ono and K.\ Soundararajan \cite{OnoSound97} (see also \cite{LemkeOliver, Rouse}). First we decompose the theta function as
\[\theta_Q(z) = E(z) + C(z)\]
where $E(z)$ is an Eisenstein series and $C(z)$ is a cusp form. Note that this decomposition can be computed quickly since the Eisenstein series is equal to a weighted sum of theta functions of the forms in the \emph{genus} $\mathcal{G}(Q)$ of $Q$ by
\[E(z) = \frac{\sum_{Q'\in \mathcal{G}(Q)} (1/|\mathrm{Aut}(Q')|) \theta_{Q'}}{\sum_{Q'\in \mathcal{G}(Q)} (1/|\mathrm{Aut}(Q')|)}.\]
Here, the genus of $Q$ is the set of ternary forms $Q'$ which are equivalent to $Q$ over the local rings $\ZZ_{p}$ for each prime $p$, as well as over $\RR$. Next, we work to understand the coefficients of $E(z)$ and $C(z).$

Letting $a_{E}(n)$ denote the Fourier coefficients of $E(z)$, it is known that if $n\geq 1$ is square-free then
        \begin{align}\label{eisenstein class}
        a_{E}(n) = \frac{24h(-nM)}{Mw(-nM)}\prod_{p \mid 2N} \beta_{p}(n)\cdot \frac{1 - \chi(p)\leg{n}{p}p^{-1}}{1 - 1/p^{2}},
        \end{align}
where $M$ is a rational number depending on $n \pmod{8N^{2}}$ with the property that $nM$ is a fundamental discriminant, $h(-nM)$ is the class number of the ring of integers in $\QQ(\sqrt{-nM})$, $w(-nM)$ is half of the number of roots of unity in $\QQ(\sqrt{-nM})$, and the $\beta_{p}(n)$ are certain local densities depending on the image of $n$ in the set
        $$\prod_{p \mid \Delta}  \QQ_{p}^{\times} / \left( \QQ_{p}^{\times}\right)^{2}.$$
Thus, for all $n$ in a fixed square class, we may write $a_E(n)=ah(-bn),$ where the constants $a$ and $b$ depend only on the square class under consideration. 

In order to study the coefficients of the cusp form $C(z),$ we will first apply the Shimura correspondence in order to obtain an integer weight modular form  (see e.g.\ \cite[Section 6]{Rouse}).
\begin{thm}
Suppose that $f(z) = \sum_{n\geq 1}a(n)q^{n} \in S_{\lambda + 1/2}(\Gamma_{0}(4N), \chi)$ is a half-integral weight cusp form with $\lambda \geq 1$. Let $t$ be a positive square-free integer and set
        $$\mathcal{S}_{t}(f(z)) := \sum_{n\geq 1} \left( \sum_{d \mid n} \chi(d)\leg{(-1)^{\lambda}t}{d} d^{\lambda - 1}a(tn^{2}/d^{2})\right)q^{n}.$$
Then $\mathcal{S}_{t}(f(z)) \in M_{2\lambda}(\Gamma_{0}(2N), \chi^{2})$. It is a cusp form if $\lambda > 1$, and if $\lambda = 1$ it is a cusp form if $f(z)$ is orthogonal to all cusp forms $\sum_{n\geq 1} \psi(n)nq^{n^{2}}$, where $\psi$ is an odd Dirichlet character. 
\end{thm}
Moreover, one can show that if $p$ is a prime not dividing $4tN$ then $\mathcal{S}_{t}(f(z) \mid T(p^{2})) = \mathcal{S}_{t}(f(z)) \mid T(p)$. Thus, if $C(z)$ is an eigenform and $2t\mid N$, this guarantees that $F(z) := \mathcal{S}_t(C(z))$ is also an eigenform with the same eigenvalues. Finally, a deep theorem of Waldspurger \cite{Waldspurger} allows us to write the Fourier coefficients of $C(z)$ in terms of the central critical $L$-values of twists of $F(z)$. By a \emph{twist} of a form $F = \sum_{n\geq 0} a(n)q^{n} \in M_{k}(\Gamma_{0}(N), \chi)$ by a Dirichlet character $\psi$ modulo $N$, we mean the modular form $F \otimes \psi := \sum_{n\geq 0} \psi(n)a(n)q^{n}$. 

\begin{thm}[Waldspurger] \label{thm:waldspurger} 
Suppose $f(z) = \sum_{n\geq 1} a(n)q^{n} \in S_{\lambda + 1/2}(\Gamma_{0}(4N), \chi)$ is a Hecke eigenform for all Hecke operators $T(p^{2})$ for primes $p \nmid N$. Let $F(z) := \mathcal{S}_{t}(f(z))$ be the Shimura lift of $f(z)$ such that $F(z) \in S_{2\lambda}^{\text{new}}(\Gamma_{0}(2N), \chi^{2})$. If $n_{1}, n_{2} \in \NN$ are square-free with $n_{1}/n_{2} \in \left( \QQ_{p}^{\times} \right)^{2}$ for all $p \mid N$, then
        $$a(n_{1})^{2}L(F \otimes \chi^{-1}\chi_{n_{2}(-1)^{\lambda}}, \lambda)\chi(n_{2}/n_{1})n_{2}^{\lambda - 1/2} = a(n_{2})^{2}L(F \otimes \chi^{-1}\chi_{n_{1}(-1)^{\lambda}}, \lambda)n_{1}^{\lambda - 1/2}.$$
\end{thm}
Thus, for all square-free $n$ in a fixed square class, we may write the Fourier coefficients of $C(z)$ as $a_C(n) = \pm dn^{1/4}L(F\otimes \chi^{-1}\chi_{-n},1)^{1/2},$ where $\chi$ is the Nebentypus character of $C(z)$ and $\chi_{-n}(\cdot)=\leg{-n}{\cdot}.$

Putting this together, for all square-free $n$ in a fixed square class, we have that \[r_Q(n) = ah(-bn) \pm dn^{1/4}L(F\otimes \chi^{-1}\chi_{-n},1)^{1/2}\]
for some constants $a,b,d$ (which depend on the square class). Since Dirichlet's class number formula gives
\begin{align}\label{DCNF}
h(-bn) = \frac{w\sqrt{bn}}{2\pi}L(\chi_{-bn}, 1),\quad w := \begin{cases} 2, & -bn < -4, \\ 4, & -bn = -4, \\ 6, & -bn = -3,  \end{cases}
\end{align}
we know that if $n$ is not represented by $Q,$ it follows that
\[\frac{L(F\otimes \chi^{-1}\chi_{-n},1)^{1/2}}{L(\chi_{-bn}, 1)} \geq \frac{a\sqrt{b}}{d\pi}n^{1/4}.\]
On the other hand, results of Chandee \cite{Chandee} can give us upper bounds for this expression. This allows us to restrict the possible values of square-free $n$ which are not represented by $Q$ to a finite set; a computer can then check these cases individually.

Integers $n$ which are not square-free must be considered using a different approach. For our particular quadratic form, Theorem \ref{thm:square-free} shows that all such integers are represented.

\section{Proof of Theorem \ref{thm:alpoge}}\label{section:alpoge}

For completeness, we now give Alpoge's proof (see Theorem 6 of \cite{Alpoge}) but correct a minor error in his calculations.

\begin{proof} By work of C.R.H Hanusa and R. Nath \cite[equation (2)]{HanusaNath}, the generating function for $sc_6(n)$ is
\begin{align*}
\sum_{n\geq 0} sc_6(n)q^{24n+35} &= \left(\frac{\eta(48z)^2}{\eta(24z)}\right)\left(\frac{\eta(288z)^3}{\eta(96z)}\right),
\end{align*}
where $\eta(z):=q^{1/24}\prod_{n\geq 1}(1-q^n)$ is the Dedekind eta function. It is known that the first factor is
\[\frac{\eta(48z)^2}{\eta(24z)} = \sum_{n\geq 0}q^{3(2n+1)^2}=\frac{1}{2}\sum_{n\in\mathbb{Z}}q^{3(2n+1)^2}\]
and the second is (\cite{HanusaNath})
\[\frac{\eta(288z)^3}{\eta(96z)} = \sum_{n\geq 0} c_3(n)q^{32(3n+1)}\]
where $c_3(n)$ is the number of 3-cores of $n.$ Work of G. Han and K. Ono \cite[Lemma 2.5]{HanOno} tells us that
\[c_3(n) = \frac{1}{6}\#\{(x,y) \in \mathbb{Z}^2: 3n+1=x^2+xy+y^2\}\]
and thus it follows that
\[sc_6(n) = \frac{1}{12} \#\{(x,y,z)\in\mathbb{Z}^3: 24n+35 = 3x^{2} + 32y^{2} + 32yz + 32z^{2}\}\]
(noting that if $24n+35=3x^2 + 32y^2 + 32yz + 32z^2$ then $x$ must be odd) as desired.
\end{proof}


\section{Initial Calculations} \label{section:initialcalcs}

In this section, we will set some notation and make some initial calculations that will be helpful in proving Theorem \ref{thm:hanusanath}. Let $Q := 3x^{2} + 32y^{2} + 32yz + 32z^{2},$ which has associated matrix
    $$A = \begin{pmatrix} 6 & 0 & 0 \\ 0 & 64 & 32 \\ 0 & 32 & 64 \end{pmatrix}.$$
Using Sage or Magma, we find that the theta function corresponding to $Q$ is
\[\theta_Q(z) = \sum_{\mathbf{x}\in\mathbb{Z}^3} q^{Q(\mathbf{x})} = \sum_{n\geq 0}r_Q(n)q^n = 1 + 2q^3 + 2q^{12} + 2q^{27} + 6q^{32} + 12q^{35} + O(q^{40}) \in M_{3/2}(\Gamma_0(96)).\] 

It is convenient that the genus of $Q$ has size 2, and the other form is $Q' := 11x^2 + 10xy + 11y^2 + 6xz - 6yz + 27z^2.$ Thus one may compute that
\begin{align*}
E(z) &= \sum_{n\geq 0} a_E(n)q^n = \frac{1}{4}\theta_Q(z) + \frac{3}{4}\theta_{Q'}(z) = 1 + \frac{1}{2} q^{3} + 3 q^{11} + 2 q^{12} + \frac{7}{2} q^{27} + 6 q^{32} + 6 q^{35} + O(q^{40}),\\
C(z) &= \sum_{n\geq 0} a_C(n)q^n = \frac{3}{4}\left(\theta_Q(z) - \theta_{Q'}(z) \right) = \frac{3}{2} q^{3} - 3 q^{11} - \frac{3}{2} q^{27} + 6 q^{35} + O(q^{40}).
\end{align*}

Applying the Shimura correspondence to $C(z)$ when $t=3$ gives a constant multiple of the newform
\[F(z) = \sum_{n\geq 0} A(n)q^n = q - q^{3} - 2q^{5} + q^{9} + 4q^{11} - 2q^{13} + 2q^{15} + 2q^{17} - 4q^{19} +O(q^{20})\in S_2(\Gamma_0(24)),\]
which is the cusp form associated to the elliptic curve $E: y^2 = x^3-x^2+x.$ 

With Theorem \ref{thm:waldspurger} in mind, we will eventually consider certain twists $F\otimes \chi_{-(24n+35)}$ of $F$. These twists translate over to the associated elliptic curve $E$ by defining the twisted curve
    $$E \otimes \chi_{-N}: y^{2} = x^{3} + Nx^{2} + N^2x,$$
where $N=24n+35.$

The following proposition will be useful in Section \ref{section:hanusanath}.

\begin{prop}\label{prop:conductor}
If $N = 24n+35$ is square-free, then the conductor of $E \otimes \chi_{-N}$ is $24N^{2}$. 
\end{prop}

\begin{proof}
We use the Weierstrass equation $y^{2} = x^{3} +Nx^{2} + N^{2}x$ for the twisted curve $E \otimes \chi_{-N},$ which has discriminant $\Delta := -48N^{6}$ (see, for example, \cite{Silverman}). The conductor $q$ of $E \otimes \chi_{-N}$ is given by
    $$q = \prod_{p \mid \Delta} p^{f_{p}},$$
where the exponents $f_p$ can be computed using Tate's algorithm (see \cite{Silverman1, Cremona}).
\end{proof}


\section{Proof of Theorem \ref{thm:hanusanath}}\label{section:hanusanath}

In order to prove Theorem \ref{maintechthm} we must first restrict our attention to values of $24n+35$ which are square-free so that we may apply Theorem \ref{thm:waldspurger} and equation \eqref{eisenstein class}. In order to do this, we proceed as in Sections 2 and 3 of \cite{OnoSound97}.

Let $R_{Q}(N)$ and $R_{Q'}(N)$ be the number of \emph{primitive} representations of $N$ by $Q$ and $Q'$, respectively. There are 12 automorphs of $Q$ (i.e., matrices $B$ of determinant 1 such that $B^TAB=B$), namely
\begin{align*}
\begin{pmatrix}-1&0&0\\0&-1&-1\\0&0&1\end{pmatrix}, \begin{pmatrix}-1&0&0\\0&-1&0\\0&1&1\end{pmatrix},& \begin{pmatrix}-1&0&0\\0&0&-1\\0&-1&0\end{pmatrix}, \begin{pmatrix}-1&0&0\\0&0&1\\0&1&0\end{pmatrix}\\
\begin{pmatrix}-1&0&0\\0&1&0\\0&-1&-1\end{pmatrix}, \begin{pmatrix}-1&0&0\\0&1&1\\0&0&-1\end{pmatrix},& \begin{pmatrix}1&0&0\\0&-1&0\\0&0&-1\end{pmatrix}, \begin{pmatrix}1&0&0\\0&-1&-1\\0&1&0\end{pmatrix}\\
\begin{pmatrix}1&0&0\\0&0&1\\0&-1&-1\end{pmatrix}, \begin{pmatrix}1&0&0\\0&0&-1\\0&1&1\end{pmatrix},& \begin{pmatrix}1&0&0\\0&1&1\\0&-1&0\end{pmatrix}, \begin{pmatrix}1&0&0\\0&1&0\\0&0&1\end{pmatrix}.
\end{align*}
Similarly, there are 4 automorphs of $Q'$, namely
\begin{align*}
\begin{pmatrix} 1 & 0 & 0 \\ 0 & 1 & 0 \\ 0 & 0 & 1 \end{pmatrix}, \begin{pmatrix} -1 & 0 & 1 \\ 0 & -1 & 1 \\ 0 & 0 & 1 \end{pmatrix}, \begin{pmatrix} 0 & -1 & 1 \\ -1 & 0 & -1 \\ 0 & 0 & -1 \end{pmatrix}, \begin{pmatrix} 0 & 1 & 0 \\ 1 & 0 & 0 \\ 0 & 0 & -1 \end{pmatrix}.
\end{align*}

We say that two representations of $N$ are \emph{essentially distinct} if there is no automorph which takes one to the other, and we let $G(N)$ be the number of essentially distinct primitive representations of $N$ by the genus of $Q$. When $N$ is square-free and coprime to $6$ one can check that
\begin{equation}\label{eq:G}
G(N) = \frac{1}{12}R_{Q}(N) + \frac{1}{4}R_{Q'}(N).
\end{equation}
Also, Theorem 86 in \cite{Jones} gives us that 
\begin{equation}\label{eq:jones}
G(N) = \frac{1}{2}h(-4N),
\end{equation}
with $-4N$ representing the discriminant rather than the determinant of the corresponding binary forms from \cite{Jones}.

\begin{thm}\label{thm:square-free}
If $N=24n+35$ is not square-free, then $r_Q(N)>0.$ 
\end{thm}

\begin{proof}
First note that if $N=24n+35$ is not square-free, then it suffices to find $d^2 \mid N$ such that $r_Q(N/d^2)>0$ (since if $Q(x_1,x_2,x_3)=N/d^2$ then $Q(dx_1,dx_2,dx_3)=N$). Thus, it suffices to prove the following statement: if $N=24n+35$ is square-free and $r_Q(N)=0,$ then $r_Q(Np^2)>0$ for any prime $p\geq 5.$

As in Section \ref{section:initialcalcs}, consider $C(z) = \frac{3}{4}(\theta_{Q}(z) - \theta_{Q'}(z)) = \sum_{n\geq 1} a_C(n)q^{n} \in S_{3/2}(\Gamma_0(96)).$ Note that $C(z)$ is a Hecke eigenform and its Shimura lift is a multiple of the newform $F(z) = \sum_{n\geq 1}A(n)q^{n} \in S_{2}(\Gamma_0(24)).$ Thus for $p\geq 5$ it follows that $A(p)$ is the Hecke eigenvalue when $T(p^2)$ is applied to $C(z),$ and so
    $$A(p)a_C(n) = a_C(p^{2}n) + \leg{-n}{p}a_C(n) + pa_C(n/p^{2}).$$
Since $a_C(n) = \frac{3}{4}(r_{Q}(n) - r_{Q'}(n))$, it follows that for square-free $n$ we have
    $$r_{Q}(np^{2}) - r_{Q'}(np^{2}) = \left( A(p) - \leg{-n}{p}\right)(r_{Q}(n) - r_{Q'}(n)).$$

Now let $p\geq 5$ be prime and $N=24n+35$ be a square-free integer such that $r_{Q}(N) = 0$. Suppose for contradiction that $r_{Q}(Np^{2}) = 0$, so
    \begin{align}\label{A(p) bound}
    \frac{r_{Q'}(Np^{2})}{r_{Q'}(N)} = A(p) - \leg{-N}{p} \leq A(p) + 1.
    \end{align}
Since $N$ is square-free, we have that
\[r_{Q'}(Np^{2}) = R_{Q'}(Np^{2}) + R_{Q'}(N) = R_{Q'}(Np^{2}) + r_{Q'}(N)\]
and (using equation \eqref{eq:G})
\[4G(N) = \frac{1}{3}R_{Q}(N) + R_{Q'}(N) = \frac{1}{3}r_{Q}(N) + r_{Q'}(N) = r_{Q'}(N).\]
Also, since $Np^{2} \neq 0$, every primitive essentially distinct representation of $Np^{2}$ by $Q'$ has at least 2 different automorphs, whence $2G(Np^{2}) \leq R_{Q'}(Np^{2})$. Therefore,
    \begin{align}\label{bound on r-quot}
    \frac{r_{Q'}(Np^{2})}{r_{Q'}(N)} = 1 + \frac{R_{Q'}(Np^{2})}{r_{Q'}(N)} \geq 1 + \frac{2G(Np^{2})}{4G(N)} = 1 + \frac{G(Np^{2})}{2G(N)}.
    \end{align}
Using equation \eqref{eq:jones} along with Corollary 7.28 from \cite{Cox}, we get
    $$\frac{G(Np^{2})}{G(N)} = \frac{h(-4Np^{2})}{h(-4N)} = p - \leg{-4N}{p} \geq p - 1.$$
Substituting this into \eqref{bound on r-quot} yields
    $$\frac{r_{Q'}(Np^{2})}{r_{Q'}(N)} \geq \frac{p+1}{2}.$$
This coupled with \eqref{A(p) bound} tells us that $A(p) + 1 \geq (p+1)/2$, whence $A(p) \geq (p-1)/2$. This contradicts Hasse's bound $|A(p)| \leq 2\sqrt{p}$ for $p>17$. Finally, there are no primes $5\leq p\leq 17$ which satisfy $A(p) \geq (p-1)/2$. This completes the proof.
\end{proof}

\begin{proof}[Proof of Theorem \ref{maintechthm}]
Let $N=24n+35$ be square-free. It suffices to show that $r_Q(N)>0$ except when $n\in \{2,12,13,73\}$.

By considering the decomposition $\theta_Q = E(z) + C(z)$ and applying both equation \eqref{eisenstein class} and Theorem \ref{thm:waldspurger}, we can find constants $a,$ $b,$ and $d$ such that
\[r_Q(N) = ah(-bN) \pm dN^{1/4}L(E\otimes \chi_{-N}, 1)^{1/2}.\]
In fact, we have $a=3$, $b=1,$ and $d=1.63384...$ Dirichlet's class number formula (equation \eqref{DCNF}) gives us
\[h(-N) = \frac{1}{\pi}\sqrt{N}L(\chi_{-N}, 1),\]
so if $N$ is not represented by $Q$ then it must be that
\begin{align}\label{lower bound}
\frac{L(E\otimes \chi_{-N}, 1)^{1/2}}{L(\chi_{-N}, 1)} = \frac{a\sqrt{b}}{d\pi}N^{1/4} \geq 0.5844N^{1/4}.
\end{align}

On the other hand, Proposition \ref{prop:conductor} allows us to utilize work of Chandee (see Section 4 of \cite{Chandee}) to obtain the upper bound
\begin{align}\label{upper bound}
\frac{L(E\otimes \chi_{-N}, 1)^{1/2}}{L(\chi_{-N}, 1)} \leq 2.5889N^{0.14157}.
\end{align}
Equations \eqref{lower bound} and \eqref{upper bound} together tell us that $n \leq 916347.7794$. That is, if $r_{Q}(N) = 0$ then $n \leq 916347.7794$. Using Sage or Magma, we find that the only such $N = 24n+35$ correspond to $n\in \{2, 12, 13, 73\}$. This completes the proof.
\end{proof}

\bibliographystyle{alpha}
\bibliography{refs}

\end{document}